\documentclass[10pt]{amsart}

\usepackage{amssymb,amsthm,amsmath}
\usepackage[numbers,sort&compress]{natbib}
\usepackage{color}
\usepackage{graphicx}
\usepackage{tikz}
\usepackage{amssymb,amsthm,amsmath}
\usepackage[numbers,sort&compress]{natbib}
\usepackage{color}
\usepackage{graphicx}
\usepackage{tikz}
\usepackage{amsmath}
\usepackage{amscd}
\usepackage{amsthm}
\usepackage{amssymb}
\usepackage{latexsym}
\usepackage{enumerate}

\title[ ]{Growth of the eigensolutions of Laplacians  on   Riemannian manifolds II: positivity of the initial energy}

\author{Wencai Liu}
\address[Wencai Liu]{Department of Mathematics, University of California, Irvine, California 92697-3875, USA}\email{liuwencai1226@gmail.com}

\theoremstyle{plain}
\newtheorem{theorem}{Theorem}[section]
\newtheorem{corollary}[theorem]{Corollary}
\newtheorem{lemma}[theorem]{Lemma}

\theoremstyle{definition}

\newtheorem{remark}[theorem]{Remark}

\begin{document}


\begin{abstract}

In this paper, energy function is used to investigate  the eigen-solutions of $-\Delta u+ Vu=\lambda u$ on the    Riemannian manifolds.
  We give  a new way  to prove the positivity of the initial energy of energy function,
  which  leads to   a simple way to obtain the
  growth of eigen-solutions.

\end{abstract}
\maketitle
\section{Introduction and main results}
The present   paper is the second  in a series \cite{LiuI} that discusses   the growth of eigen-solutions  of Laplacian on connected  non-compact complete Riemannian manifolds.

Let $(M,g)$ be a   connected  $n$-dimensional noncompact complete Riemannian manifold ($n\geq 2$).
The Laplace-Beltrami operator   on   $(M,g)$ is essentially self-adjoint on $C^{\infty}_0(M)$.  
We  denote the self-adjoint extension by $\Delta$ (Laplacian).
Assume there exists $U\subset M$ such that  $M\backslash U$  is  connected
 and the induced outward normal exponential map $\exp_{\partial U}^{\perp} : N^{+}(\partial U) \to M - \overline{U}$   is diffeomorphism, where $N^{+}(\partial U) = \{ v\in T(\partial U) \mid v {\rm ~is~outward~normal~to~}\partial U \}$.
Let $r$ be the distance function from $\partial U$ defined on the end $M-\overline{U}$.
This is the same  setting  as in \cite{kumura2010radial,kumuraflat,LiuI}.

An interesting problem  is to investigate  if the   eigenvalues  can be embedded into the essential spectrum of   Laplacian on Riemannian manifolds.
We  focus on two popular manifolds-asymptotically flat case and asymptotically hyperbolic case, which are distinguished   by the radial curvature $K_{\rm rad}(r)$. Radial curvature is a very important quantity in spectral geometry \cite{greene2006function}.
Roughly speaking, asymptotically flat manifold is characterized by  the radial curvature $K_{\rm rad}(r)= O(r^{-\alpha})$ with some $\alpha>0$, and respectively asymptotically hyperbolic manifold  is characterized by the radial curvature $K_{\rm rad}(r)=-1+O(r^{-\alpha})$.
This topic was heavily studied by various authors in geometry and analysis  \cite{ISad,ISI,ISII,ISJFA,donnelly1990negative,pinsky1979spectrum,kumuraflat,donnelly1979pure,kumura1997essential,donnelly1992,donnelly1999,escobar1992spectrum,MR2951504,MR2914877,MR3100776}.

For the asymptotically hyperbolic case,
the sharp spectral  transition is given by Kumura   \cite{kumura2010radial}.
 He   excludes eigenvalues greater than $\frac{(n-1)^2}{4}$ under the assumption that $K_{\rm rad}(r)  =-1+ o(r^{-1})$, and  also  constructs a manifold
  for which exact  an eigenvalue $\frac{(n-1)^2}{4} + 1$ is embedded  into its essential spectrum
 $[ \frac{(n-1)^2}{4} , \infty )$ with the radial curvature  $K_{\rm rad}(r)  = -1+O(r^{-1})$.
Jitomirskaya and Liu   constructed examples which show that dense eigenvalues  and singular continuous spectrum can be embedded into the essential spectrum \cite{jl1,jl2}.

 For the asymptotically flat case,
  the authors in \cite{donnelly1992,donnelly2010spectral,donnelly1999,escobar1992spectrum} showed the absence of positive eigenvalues of    Laplacian
 on  complete Riemannian manifold with a pole, under  the   assumptions  more or less like   $|K_{\rm rad}(r)|   \leq \frac{\delta}{1+r^2}$ for small $\delta$.
Kumura  obtained a similar result (see Remark \ref{reku}) which works for more general manifolds \cite{kumuraflat}.

We should mention that the spectral theory of free Laplacian on asymptotically hyperbolic and flat manifolds  is more or less parallel to  that of Schr\"odinger operators with decaying potentials \cite{MR2307748}. We refer the readers to our previous paper \cite{LiuI}.

Kumura \cite{kumura2010radial,kumuraflat} achieved his goals  by the careful study of  energy function of  eigen-solutions directly. Our inspiration comes from Kato \cite{kato}.
 His basic  idea is (four steps):  construct  energy function for eigen-equation; set up the positivity of initial energy; prove the monotonicity of energy function with respect to $r$; obtain the growth of eigen-solution.
In our previous paper \cite{LiuI}, we have already used such scheme to set up the growth of eigen-solutions by  a universal construction of energy functions.
As an application, we showed the absence of   eigenvalues in certain circumstances, which improved Kumura's sharp results \cite{kumura2010radial} mentioned before.
 However, we  rely on strong convexity of Hessian of distance function $r$  to set up the positivity (or non-negativity) of the initial energy.
  Such strong convexity  is a big obstacle for     the asymptotically flat manifolds, which prevents  us from getting  sharp bounds.
  Since the energy function on the manifold (even on the Euclidean space in higher dimensions) is not the simple sum of potential energy and kinetic energy any more, the non-negativity of $-\Delta$ can not lead to the positivity (or non-negativity) of initial energy.
 Notice that the positivity of initial energy  for  simply connected manifolds is trivial because the energy function we constructed is  zero at the initial point $r=0$ automatically. This explains why it is  easier to study the spectral theory of Laplacian on simply connected manifolds in some sense.

 The novelty of this paper is to give an effective  and new way to verify the   positivity (or non-negativity) of the initial energy for non-simply connected manifolds.
  The idea   is to increase the potential energy via modifying the potential without breaking the scheme (four steps) in \cite{LiuI,kato}. 
 We show that if the  positivity of the initial energy  fails, by solving a ODE problem,  the weakly exponential decay  ($e^{-c\sqrt{r}}$ with some $c>0$) of the   eigen-solution can be established.
 By using a class of test functions introduced  in \cite{kumuraflat} and   developing the  methods of \cite{kumuraflat}, we prove that $u$ is  exponentially decaying ($e^{-c {r}}$ with some $c>0$) and then is a trivial solution by additional arguments.
 As one of the  applications, the results in \cite{kumuraflat} are recovered without the extra assumptions on the Ricci curvature.
 To  our best  knowledge, this is a new way to prove the positivity of initial energy.
 We only give one kind of such theorems here.
 The gradient version and mixed version   can also be established   by combining with the arguments
 in \cite{LiuI}.
 Moreover, we believe our method has a  wider applicability.

Another approach to investigate the absence of embedded eigenvalues is based on Mourre type
commutator estimates and super-exponential decay estimates of a priori eigenfunction. See \cite{ISad} and references therein.
Ito and Skibsted obtained some criteria for the absence of embedded  eigenvalues  for both asymptotically flat and hyperbolic Riemannian
Laplacians \cite{ISad}, which are similar to the  results in the present paper and our    previous paper \cite{LiuI}. We will say more later.
Ito and Skibsted also set up the scattering theory for Riemannian Laplacians in their series of papers \cite{ISI,ISII,ISJFA,ISIII}.

Let $g$ be the metric and  $\nabla$  be  the covariant derivative. Denote
Hessian of $r$ by $\nabla dr$. 
 Let
 \begin{eqnarray*}
   M(r) &=& M(r;u) =(\int_{|r(x)|=r}|u(x)|^2 dx)^{\frac{1}{2}},\\\label{GM}
   N(r) &=& N(r;u)=(\int_{|r(x)|=r}|\frac{\partial  u}{\partial r}|^2 dx)^{\frac{1}{2}}. \label{GN}
 \end{eqnarray*}

Our main result in the present paper is
\begin{theorem}\label{Mainthm1}
Let  the potential  $V(r)=V_1(r)+V_2(r)$. Suppose
\begin{equation*}
    |V_1(r)|=\frac{o(1)}{r}, |V_2(r)|=o(1), |\frac{\partial V_2}{\partial r}|=\frac{o(1)}{r},
\end{equation*}
as $r$ goes to infinity.
Suppose
\begin{equation}\label{ajun}
   \liminf_{r\to \infty}[  r \nabla dr-a\hat{g}]\geq 0,
\end{equation}
for some $a>0$,
where $ \hat{g}=g-dr\otimes dr$, and
\begin{equation*}
 \limsup_{r\to \infty}  r|\Delta r-b-\frac{c}{r}|\leq  \delta,
\end{equation*}
 for some   constants $b,c,\delta$.
Suppose
\begin{equation}\label{Gcons}
    \mu>\delta,2a > \mu+\delta, 
\end{equation}
and
\begin{equation}\label{Gconl}
   \lambda> \frac{b ^2}{4}+ \frac{ {\delta}^2 b ^2}{\mu^2-{\delta}^2}.
\end{equation}
Then we have
\begin{equation*}
   \liminf_{r\to \infty} r^{\mu} [M(r)^2+N(r)^2] = \infty .
\end{equation*}

\end{theorem}
Based on the universal Theorem \ref{Mainthm1}, we have  several interesting   corollaries.
Letting  $b=0$ in Theorem \ref{Mainthm1}\footnote{By the fact that $\Delta+V$ is essentially selfadjoint, we have $\nabla u\in L^2(M)$ if the eigensolution $u\in L^2(M)$.},   we have
\begin{corollary}\label{Maincor1}
Let  the potential  $V(r)=V_1(r)+V_2(r)$. Suppose
\begin{equation*}
    |V_1(r)|=\frac{o(1)}{r}, |V_2(r)|=o(1), |\frac{\partial V_2}{\partial r}|=\frac{o(1)}{r},
\end{equation*}
as $r$ goes to infinity.
Suppose
\begin{equation*}
   \liminf_{r\to \infty}[  r \nabla dr-a\hat{g}]\geq 0,
\end{equation*}
for some $a>0$,
where $ \hat{g}=g-dr\otimes dr$, and
\begin{equation*}
 \limsup_{r\to \infty}  r|\Delta r-\frac{c}{r}|\leq  \delta,
\end{equation*}
 for some   constants $c,\delta$.
Suppose
\begin{equation*}
   \delta<\min\{a,1\}. 
\end{equation*}
Then  $-\Delta+V$ admits no positive eigenvalue.

\end{corollary}


\begin{corollary}\label{Cor1}
Suppose
\begin{align*}
   \frac{a}{r}\hat{g}
   \leq \nabla dr \leq
   \frac{b}{r} \hat{ g }
\end{align*}
for large $r$,
where $a >0$ and $b>0$ are constants satisfying
\begin{align}\label{Gequ1thm21}
   a \le b \quad {\rm and} \quad \frac{n+1}{n-1}a > b . 
\end{align}
Let $\lambda > 0$  and $u$ be a nontrivial solution to the eigen-equation
\begin{align*}
     \Delta u + \lambda u = 0 .
\end{align*}
Then  for any
\begin{align*}
   \mu > \frac{n-1}{2} ( b - a ), 
\end{align*}
we have
\begin{align*}
   \liminf_{r \to \infty}~r^{\mu}[ M(r)^2+N(r)^2] =\infty.
\end{align*}
\end{corollary}
\begin{proof}
By the fact that $\Delta r$ is the trace of $\nabla dr$,
one has
\begin{equation*}
  \frac{(n-1)a}{r} \leq \Delta r\leq \frac{(n-1)b}{r}.
\end{equation*}
So
\begin{equation*}
  |\Delta r-\frac{(b-a)(n-1)}{2r}|\leq \frac{\delta}{r},
\end{equation*}
where $\delta=\frac{(b-a)(n-1)}{2}$.
Let $s$ satisfy   $\mu>s>\delta$ and  be close to $\delta$ so that
\begin{equation*}
    2a>\delta+s
\end{equation*}
  by \eqref{Gequ1thm21}.
By Theorem \ref{Mainthm1}, one has
\begin{equation*}
    \liminf_{r\to \infty} r^s[M(r)^2+N(r)^2]=\infty,
\end{equation*}
which implies
 \begin{equation*}
    \liminf_{r\to \infty} r^{\mu}[M(r)^2+N(r)^2]=\infty.
\end{equation*}
\end{proof}

Letting $\mu =1$ in Corollary \ref{Cor1}, we can get the following result.
\begin{corollary}\label{Cor2}
Suppose
\begin{align*}
   \frac{a}{r} \hat{g}
   \leq \nabla dr \leq
   \frac{b}{r} \hat{g}
\end{align*}
for large $r$,
where $a >0$ and $b>0$ are constants satisfying
\begin{align*}
   a \leq b \quad {\rm and} \quad \frac{n+1}{n-1}a > b .
\end{align*}
Suppose
\begin{equation*}
  \frac{n-1}{2} ( b - a )<1.
\end{equation*}
Then $\sigma(-\Delta)=\sigma_{\rm ess }(-\Delta)=[0,\infty)$ and $\sigma_{\rm p }(-\Delta)=\emptyset$.
\end{corollary}
Similarly, we have
\begin{corollary}\label{Corjun}
Let $\kappa>0$.
Suppose
\begin{align*}
  (\kappa+ \frac{a}{r}) \hat{g}
   \leq \nabla dr \leq
   (\kappa+\frac{b}{r} )\hat{g}
\end{align*}
for large $r$,
where $a >0$ and $b>0$ are constants satisfying
\begin{align*}
   a \leq b \quad {\rm and} \quad (n-1)(b-a)<2 .
\end{align*}
Let $$E_0=\frac{ \kappa^2(n-1)^2}{4}+\frac{(n-1)^4\kappa^2(b-a)^2}{4-(n-1)^2(b-a)^2}.$$
Then  $-\Delta$ does not have eigenvalue larger than $ E_0
    $.
\end{corollary}
\begin{remark}\label{reku}
\begin{itemize}
\item
We used the fact that
$ \sigma_{\rm ess}(-\Delta)=[\frac{c^2}{4},\infty)$ if
$\lim \Delta r =c$   \cite{donnelly1979pure,kumura1997essential} and $-\Delta$ is a non-negative operator.
\item Corollaries \ref{Cor1} and \ref{Cor2} were obtained by
Kumura \cite{kumuraflat} with some extra assumptions on Ricci curvature.
\end{itemize}
\end{remark}
A result similar to Theorem   \ref{Mainthm1}  has been already obtained in \cite{LiuI,ISad} and similar   Corollaries \ref{Cor1} and \ref{Cor2}  have also been already obtained.
Here, we want to  say more about  the theorems in \cite{LiuI,ISad} and the present paper.
We will use  Corollaries \ref{Cor2} and \ref{Corjun} as an example first.
Under the same assumptions of  Corollary \ref{Corjun}, it was shown
 $-\Delta$ does not have eigenvalue larger than $
    E_2$ (\cite[Corollary 2.4]{ISad}), where
    $$E_2=\frac{ \kappa^2(n-1)^2}{2(2-(n-1)(b-a))}.$$
     Under the same assumptions of  Corollary \ref{Corjun}, it was shown
 $-\Delta$ does not have eigenvalue larger than $
    E_1$ (see Corollary 1.3 and Remark 1.4 in \cite{LiuI}), where
    $$E_1=\frac{ \kappa^2(n-1)^2}{4}+\frac{(n-1)^4\kappa^2(b-a)^2}{4(4-(n-1)^2(b-a)^2)}.$$
 Thus we have
\begin{description}
  \item [Corollary \ref{Cor2} ] For $\kappa=0$, the assumptions in the present paper are stronger than those in \cite{ISad}. Thus the result  in \cite{ISad} is best.
  \item  [Corollary \ref{Corjun} ] For $\kappa>0$, all the assumptions in  \cite[Theorem 1.5]{ISad}, \cite[Theorem 1.1]{LiuI} and  present paper are exactly the same.
  The bound $E_1$ in \cite{LiuI} is the best, that is $E_1\leq E_0$ and $E_1\leq E_2$. If $b-a$ is small (large), the bound $E_0$ in the present paper is better (worse) than that in
  \cite{ISad}.
\end{description}

 It is not just the corollaries in the three papers are comparable,
 the main theorems-\cite[Theorem 1.5]{ISad}, \cite[Theorem 1.1]{LiuI} and Theorem \ref{Mainthm1} are also comparable.
 Theorem 1.5 in \cite{ISad} was formulated in  a different way, which is involved into the derivatives of $\Delta r$ and $r(x)$ is not necessary to be the distance function.
 By direct modifications of our arguments, a similarly generalized version  may also be obtained.

Suppose $b>0$.
 Theorem \ref{Mainthm1} implies (letting $\mu=1$), if $\delta<1$,
 $-\Delta$ does not have eigenvalue larger than $
    E_0$, where
$$E_0=   \frac{b ^2}{4}+ \frac{ {\delta}^2 b ^2}{1-{\delta}^2}.$$
Under the assumptions of Theorem \ref{Mainthm1}, Ito and Skibsted obtained that \footnote{Let $c_1=2-2\delta$, $\rho_1=2(\delta+ r(\Delta r-b-\frac{c}{r}))$, $\rho_2=0$, $\rho_3=2+2rb+2c-2\delta$ and $c_2>0$ be  arbitrarily  small in Theorem 1.5 in \cite{ISad}. Such setting is optimal under the assumptions of Theorem \ref{Mainthm1}.}
if $\delta<1$,
 $-\Delta$ does not have eigenvalue larger than $
    E_2$
$$E_2=   \frac{b ^2}{4(1-\delta)}.$$
Under the assumptions of Theorem \ref{Mainthm1} with $a\geq 2$, the bound obtained in \cite{LiuI} is best among the three papers. More precisely, $-\Delta$ does not have eigenvalue larger than $
    E_1$
$$E_1=   \frac{b ^2}{4}+ \frac{ {\delta}^2 b ^2}{4(1-{\delta}^2)}$$
for $\delta<1$.

Without the extra assumption $a\geq 2$, whether  the bound in \cite{ISad} or in the present paper is better depends on the size of the  perturbation $\delta$ ($\frac{1}{3}$ is the threshold).

Letting  $b=0$ in Theorem \ref{Mainthm1}, we have
\begin{corollary}\label{Cor3}
Suppose
\begin{equation*}
     \nabla dr=(\frac{a}{r}+\frac{o(1)}{r})\hat{g},
\end{equation*}
for some $a>0$.
Let the potential  $V(r)=V_1(r)+V_2(r)$. Suppose
\begin{equation*}
    |V_1(r)|=\frac{o(1)}{r}, |V_2(r)|=o(1), |\frac{\partial V_2}{\partial r}|=\frac{o(1)}{r}.
\end{equation*}
Then   operator $-\Delta+V$ does not have positive eigenvalue.
\end{corollary}
\begin{remark}
 Corollary \ref{Cor3} works for in particular  Euclidean space ($\nabla dr=\frac{1}{r}\hat{g}$),  covering a well known result.
\end{remark}

\begin{corollary}\label{Cor4}
Let $M^n$ be a complete Riemannian manifold with a pole. Suppose
the radial curvature satisfies $K_{\rm rad}(r)\leq \frac{\delta_n}{1+r^2}$, for  sufficiently small $\delta_n$. Then M
admits no eigenvalue.
\end{corollary}
\begin{remark}
\begin{itemize}
\item
Corollary \ref{Cor4} was proved by Donnelly and Garofalo \cite{donnelly1992}. A weaker version (curvature condition $0\leq K_{\rm rad}(r)\leq \frac{\delta_n}{1+r^2}$)
was  proved by Escobar-Freire \cite{escobar1992spectrum}.
Corollary \ref{Cor4} follows from Theorem \ref{Mainthm1} and comparison theorem. We gave up the details here.
See Lemma 1.2 \cite{escobar1992spectrum}, and Propositions 2.1, 2.2 in \cite{kumuraflat} for comparison theorem.
\item
By checking our proof,
Corollary \ref{Cor4} also holds if we  replace the  radial curvature  with  Ricci curvature of the direction $\nabla r$.
\end{itemize}
\end{remark}
Now the only task is  to prove Theorem
\ref{Mainthm1}.
The rest of the paper is organized as follows:
In \S 2, we present some basic knowledge.
In \S 3, we will give  the  construction of energy functions and derive their derivatives.
In \S 4, we prove the exponential decay of the eigensolution.
In \S 5, we establish the  positivity of the initial energy corresponding to the energy function  and  give the proof    of Theorem \ref{Mainthm1}.

\section{Preliminaries and derivative lemma}
We recall some  notations in \cite{LiuI} first.
Let   $S_t=\{x\in M: r(x)=t\}$,  $ \omega\in  S_r$ and $x\in M$. Thus   $(r,\omega)$ is a local coordinate system for $M$ \footnote{$\omega$ depends on $r$. We ignore the dependence for simplicity.}.
Let $\langle\cdot,\cdot\rangle$ be  the metric on Riemannian manifold. 


 Let  $\hat{L}=e^{\rho}L e^{-\rho}$, where $L=-\Delta+V$.
 We draw the graph to describe the relation.
 \begin{align*}
  \begin{CD}
    L^2(M,dg)           @>{ -\Delta + V}>>       L^2(M,dg)          \\
    @V{e^{\rho}}VV                         @VV{e^{\rho}}V  \\
    L^2(M,e^{-2\rho}dg)   @>>{ \hat{ L} }>           L^2(M,e^{-2\rho}dg)
  \end{CD}
\end{align*}
Let
\begin{equation*}
    v=e^{\rho}u.
\end{equation*}
Then, one has
\begin{equation*}
    \nabla u=-\rho^{\prime}e^{-\rho}v \nabla r+e^{-\rho} \nabla v,
\end{equation*}
and
\begin{equation*}
    \Delta u={\rm div} \nabla u=e^{-\rho}\Delta v -2\rho^{\prime}e^{-\rho}\frac{\partial v}{\partial r}+(\rho^{\prime 2}-\rho^{\prime \prime}-\rho^{\prime} \Delta r)e^{-\rho}v.
\end{equation*}
So  the eigen-equation becomes
\begin{equation}\label{equav}
    -\Delta v+2\rho^{\prime}\frac{\partial v}{\partial r}+(V_1+V_2+V_0)v=\lambda v,
\end{equation}
where
\begin{equation*}
V_0=\rho^{\prime} \Delta r+\rho^{\prime \prime}-\rho^{\prime 2}.
\end{equation*}

Let $2\rho^{\prime}=b+\frac{c}{r}$ and  $r(\Delta r-b-\frac{c}{r})=\bar{\delta}(r)$.
 Then
\begin{equation*}
V_0=\frac{b^2}{4}+\frac{bc}{2r}+\frac{b\bar{\delta}}{2r}+\frac{O(1)}{r^2}.
\end{equation*}

We should mention that $O(1)$ and $o(1)$  in the proof only depend  on constants in the  assumptions of  Theorem \ref{Mainthm1}.

Now we always assume $u$ is a nontrivial solution of $-\Delta u+Vu=\lambda u$, where $V=V_1+V_2$.
Let
\begin{equation*}
   v=e^{\rho}u,
\end{equation*}
and
\begin{equation*}
    v_m=r^m v
\end{equation*}
with $m\geq 0$.
By \eqref{equav}, we get the equation of $v_m$,
\begin{equation}\label{equav1new}
    \Delta v_m -(\frac{2m}{r}+2\rho^{\prime} )\frac{\partial v_m}{\partial r}+(\frac{m(m+1)}{r^2}+\frac{m}{r}(2\rho^{\prime}-\Delta r)-V_0-V_1-V_2+\lambda)v_m=0.
\end{equation}
\begin{lemma}\label{Keylemma1}\cite{kumura2010radial,LiuI}
Let $X$ be a vector field. Then
\begin{equation}\label{Gder1}
    \frac{\partial}{\partial r}\int_{S_r}\langle X,\nabla r\rangle e^{-2\rho}dx=\int_{S_r}({\rm div} X -2\rho^{\prime}\langle X,\nabla r\rangle) e^{-2\rho}dx.
\end{equation}

\end{lemma}

\begin{lemma}\label{Keylemma2}\cite{kumura2010radial,LiuI}
\begin{equation*}
    \frac{\partial}{\partial r}\int_{S_r}f  e^{-2\rho}dx=\int_{S_r} [\frac{\partial f}{\partial r}+f(\Delta r-2\rho^{\prime})]e^{-2\rho}dx.
\end{equation*}
\end{lemma}

\begin{lemma}\label{Keylemma3}
 Let $q_0=\lambda-V_0-V_1-V_2+\frac{m(m+1)}{r^2}+\frac{m}{r}(2\rho^{\prime}-\Delta r)$.
 Suppose
 \begin{equation}\label{Gequ834}
  \liminf_{r\to\infty}  \int_{S_{r}}r^{1-2m}|\frac{\partial v_m}{\partial r}v_m|e^{-2\rho}dx=0.
\end{equation}
 Then
 the following holds
 \begin{eqnarray*}
   \int_{r >t} r^{1-2m}[|\nabla v_m|^2-q_0v_m^2]e^{-2\rho} dx&=& -\frac{1}{2}\frac{d}{dt}(t^{1-2m}\int_{S_t}v_m^2e^{-2\rho}dx)-\frac{1}{2}\int_{S_t}r^{-2m}(2m-1)v_m^2e^{-2\rho}dx \\
    &&+\frac{1}{2}\int_{S_t}t^{1-2m} (\Delta r-2\rho^{\prime})v_m^2e^{-2\rho} dx-\int_{r>t}r^{-2m}\frac{\partial v_m}{\partial r}v_me^{-2\rho}dx.
 \end{eqnarray*}
 \end{lemma}
 \begin{proof}
 By  \eqref{Gequ834}, there exists a sequence $r_j$ going to infinity such that
 \begin{equation}\label{Gequ834new}
   \int_{S_{r_j}}r^{1-2m}|\frac{\partial v_m}{\partial r}v_m|e^{-2\rho}dx=o(1).
\end{equation}
By Lemma \ref{Keylemma2}, one has
\begin{eqnarray}
\nonumber   -\frac{1}{2}\frac{d}{dt}(t^{1-2m}\int_{S_t}v_m^2e^{-2\rho}dx) &=& -\frac{1}{2}(1-2m)t^{-2m}\int_{S_t}v_m^2e^{-2\rho}dx-\frac{1}{2}t^{1-2m}\int_{S_t}(\Delta r -2\rho^{\prime})v_m^2e^{-2\rho}dx \\
  && -  t^{1-2m}\int_{S_t}\frac{\partial v_m}{\partial r} v_me^{-2\rho}dx.\label{Glemma31}
\end{eqnarray}
Multiplying $r^{1-2m}v_m e^{-2\rho}$ on both sides of eigen-equation \eqref{equav1new} and integration by part, one has
\begin{eqnarray*}
  \int_{t\leq r\leq r_j} r^{1-2m}q_0v_m^2e^{-2\rho}dx &=& \int_{t\leq r\leq r_j} (1-2m)r^{-2m}\frac{\partial v_m}{\partial r}v_me^{-2\rho}dx \\
   && +  \int_{t\leq r\leq r_j} r^{1-2m} |\nabla v_m|^2   e^{-2\rho}dx+\int_{t\leq r\leq r_j}2mr^{-2m} \frac{\partial v_m}{\partial r}v_m e^{-2\rho}dx\\
    &&+\int_{S_{t}} r^{1-2m}\frac{\partial v_m}{\partial r}v_m e^{-2\rho}dx-\int_{S_{r_j}} r^{1-2m}\frac{\partial v_m}{\partial r}v_m e^{-2\rho}dx.
\end{eqnarray*}
By \eqref{Gequ834new}, one has
\begin{eqnarray}
 \nonumber \int_{ r> t} r^{1-2m}q_0v_m^2e^{-2\rho}dx &=& \int_{  r>t} r^{-2m}\frac{\partial v_m}{\partial r}v_me^{-2\rho}dx \\
   && +  \int_{  r>t} r^{1-2m} |\nabla v_m|^2   e^{-2\rho}dx+\int_{S_{t}} r^{1-2m}\frac{\partial v_m}{\partial r}v_m e^{-2\rho}dx.\label{Glemma32}
\end{eqnarray}
Now the lemma follows by \eqref{Glemma31} and \eqref{Glemma32}.
 \end{proof}

\section{Construction of the energy functions}

In this section, we will give the general construction of energy functions and derive the formulas for their derivatives.
Those derivation can be found in \cite{LiuI}. In order to make the readers easy to understand, we  give the details here.
Let
\begin{eqnarray}
  F(m,r,s) &=& r^s\int_{S_r}\frac{1}{2}[\frac{m(m+1)}{r^2}+q]v_m^2e^{-2\rho}dx \label{Gdefenergy}\\
 \nonumber &&+r^s\int_{S_r}[|\frac{\partial v_m}{\partial r}|^2-\frac{1}{2}|\nabla v_m|^2)]e^{-2\rho}dx \\
  \nonumber &=& {\rm I} +{\rm II},
\end{eqnarray}
where
\begin{eqnarray*}
  {\rm I} &=& r^s\int_{S_r}[|\frac{\partial v_m}{\partial r}|^2-\frac{1}{2}|\nabla v_m|^2]e^{-2\rho}dx \\
   &=& \frac{1}{2} r^s\int_{S_r}[|\frac{\partial v_m}{\partial r}|^2-|\nabla_{\omega} v_m|^2]e^{-2\rho}dx ,
\end{eqnarray*}
and
\begin{equation*}
   {\rm II}= \frac{1}{2}r^s\int_{S_r}[\frac{m(m+1)}{r^2}+q]v_m^2 e^{-2\rho}dx.
\end{equation*}
We begin with the derivation of $\frac{\partial }{\partial r}{\rm I}$.
By Lemma \ref{Keylemma2},  one has
\begin{eqnarray*}
  \frac{\partial }{\partial r}{\rm I} &=& sr^{s-1}\int_{S_r }[\frac{1}{2}|\frac{\partial v_m}{\partial r}|^2-\frac{1}{2}|\nabla_{\omega} v_m|^2]e^{-2\rho}dx+r^s\int_{S_r}[\frac{\partial v_m}{\partial r}\frac{\partial^2v_m}{\partial r^2}-\frac{1}{2} \frac{\partial }{\partial r}\langle\nabla_{\omega} v_m,\nabla_{\omega}v_m\rangle]e^{-2\rho} dx\\
   && +r^s\int_{S_r}  \frac{1}{2}(\Delta r-2\rho^{\prime})[|\frac{\partial v_m}{\partial r}|^2-|\nabla_{\omega} v_m|^2]e^{-2\rho}dx.
\end{eqnarray*}
Using $ \Delta v_m=\frac{\partial^2v_m}{\partial r^2}+\Delta r\frac{\partial v_m}{\partial r}+\Delta_{\omega} v_m$, we get
\begin{eqnarray*}
  \frac{\partial }{\partial r}{\rm I} &=& \int_{S_r}[\frac{s}{2}r^{s-1}|\frac{\partial v_m}{\partial r}|^2-2r^s\rho ^{\prime}|\frac{\partial v_m}{\partial r}|^2+r^s\frac{\partial v_m}{\partial r} \Delta v_m+ \frac{r^s}{2}(2\rho^{\prime}-\Delta r)|\frac{\partial v_m}{\partial r}|^2 ]e^{-2\rho}dx\\
   &&+
   r^{s}\int_{S_r}[(  -\frac{s}{2r }+\frac{1}{2}(2\rho^{\prime}-\Delta r))\hat{g}(\nabla v_m,\nabla v_m)]e^{-2\rho}dx\\
   &&+r^s\int_{S_r}[\langle\nabla_{\omega} \frac{\partial v_m}{\partial r},\nabla_{\omega}v_m\rangle-\frac{1}{2} \frac{\partial }{\partial r}\langle\nabla_{\omega} v_m,\nabla_{\omega}v_m\rangle]e^{-2\rho} dx.
\end{eqnarray*}
By some basic computations, one has
\begin{eqnarray*}
  \langle\nabla_{\omega} \frac{\partial v_m}{\partial r},\nabla_{\omega}v_m\rangle-\frac{1}{2} \frac{\partial }{\partial r}\langle\nabla_{\omega} v_m,\nabla_{\omega}v_m\rangle &=& \langle\nabla_{\omega} \frac{\partial v_m}{\partial r},\nabla_{\omega}v_m\rangle- \langle\nabla _{ \frac{\partial }{\partial r}}\nabla_{\omega} v_m,\nabla_{\omega}v_m\rangle \\
   &=& (\nabla dr)(\nabla_{\omega}v_m,\nabla_{\omega}v_m) .
\end{eqnarray*}
Finally, we get
\begin{eqnarray}
 \nonumber \frac{\partial }{\partial r}{\rm I} &=& \int_{S_r}[\frac{s}{2}r^{s-1}|\frac{\partial v_m}{\partial r}|^2-2r^s\rho ^{\prime}|\frac{\partial v_m}{\partial r}|^2+r^s\frac{\partial v_m}{\partial r} \Delta v_m+ \frac{r^s}{2}(2\rho^{\prime}-\Delta r)|\frac{\partial v_m}{\partial r}|^2 ]e^{-2\rho}dx\\
   &&+
   r^{s}\int_{S_r}[( \nabla dr +(-\frac{s}{2r }+\frac{1}{2}(2\rho^{\prime}-\Delta r))\hat{g})(\nabla v_m,\nabla v_m)]e^{-2\rho}dx.\label{GpartialI}
\end{eqnarray}
Now we are in the position to obtain $ \frac{\partial}{\partial r}{\rm II}$.
By Lemma \ref{Keylemma2} again, one has
\begin{eqnarray}
  \nonumber  \frac{\partial}{\partial r}{\rm II}&= &\int_{S_r}  [\frac{\partial}{\partial r} \frac{r^s}{2}(\frac{m(m+1)}{r^2} +q)v_m^2]e^{-2\rho}dx \\
 \nonumber   &&+\int_{S_r}(\Delta r-2\rho^{\prime})\frac{r^s}{2}[\frac{m(m+1)}{r^2}+q]v_m^2e^{-2\rho}dx\\
 \nonumber   &=& \int_{S_r}  [ \frac{s-2}{2} r^{s-3}m(m+1) +\frac{r^s}{2}\frac{\partial q}{\partial r}+\frac{s}{2}r^{s-1}q]v_m^2 e^{-2\rho}dx\\
 \nonumber   &&+r^s \int_{S_r}[ \frac{m(m+1)}{r^2}+q]v_m\frac{\partial v_m}{\partial r} e^{-2\rho}dx\\
   &&+\frac{r^s}{2}\int_{S_ r}(\Delta r-2\rho^{\prime})[\frac{m(m+1)}{r^2}+q]v_m^2e^{-2\rho}dx.\label{GpartialII}
\end{eqnarray}
Putting \eqref{GpartialI} and \eqref{GpartialII} together  and using \eqref{equav1new}, we obtain
\begin{eqnarray}
 \label{Gpartial} \frac{\partial F(m,r,s)}{\partial r} &=&r^{s-1}\int_{S_r} [( r(\nabla dr) -(\frac{s}{2 }+\frac{\bar{\delta}}{2})\hat{g})(\nabla v_m,\nabla v_m)]e^{-2\rho} dx\\
   \nonumber &&+r^{s-1}\int_{S_r} [2m -\frac{\bar{\delta}}{2}+\frac{s}{2}]|\frac{\partial v_m}{\partial r}|^2 e^{-2\rho}dx\\
   \nonumber &&+r^{s-1}\int_{S_r}[r (V_0+V_1+V_2+q-\lambda)+\frac{m}{r}\bar{\delta}]\frac{\partial v_m}{\partial r}v_m  e^{-2\rho}dx\\
  \nonumber  && +r^{s-1}\int_{S_r} [\frac{s-2}{2} \frac{m(m+1)}{r^2}+\frac{1}{2}r\frac{\partial q}{\partial r}+\frac{s}{2}q]v_m^2e^{-2\rho}dx\\
  \nonumber  &&+r^{s-1}\int_{S_r}\frac{1}{2}\bar{\delta}[\frac{m(m+1)}{r^2}+q]v_m^2e^{-2\rho}dx.
\end{eqnarray}
Let $m=0$ and $F(r,s)=F(0,r,s)$, we obtain

\begin{eqnarray}
\label{Gpartialnew}\frac{\partial F(r,s)}{\partial r} &=&r^{s-1}\int_{S_r}  [ (r(\nabla dr) -(\frac{s}{2 }+\frac{1}{2}\bar{\delta})\hat{g})(\nabla v,\nabla v)]e^{-2\rho} dx\\
   \nonumber &&+r^{s-1}\int_{S_r} [-\frac{\bar{\delta}}{2}+\frac{s}{2}]|\frac{\partial v}{\partial r}|^2 e^{-2\rho}dx\\
   \nonumber &&+r^{s-1}\int_{S_r}[r (V_0+V_1+V_2+q-\lambda)]\frac{\partial v}{\partial r}v  e^{-2\rho}dx\\
  \nonumber  && +r^{s-1}\int_{S_r} [\frac{1}{2}r\frac{\partial q}{\partial r}+\frac{s}{2}q+\frac{\bar{\delta}}{2}q]v^2e^{-2\rho}dx.
\end{eqnarray}
\section{Exponential  decay of the eigensolutions}
In this section, we will prove the exponential  decay of the eigensolution based on the weakly  exponential  decay assumption and then show such  eigensolution  is trivial.
\begin{theorem}\label{Expthm1}
Suppose
\begin{equation*}
   \int_{r \geq r_0}(|\nabla u|^2+u^2)e^{c\sqrt{r}}dx<\infty,
\end{equation*}
for some $c>0$,
then
\begin{equation*}
   \int_{r \geq r_0}(|\nabla u|^2+u^2)e^{c {r}}dx<\infty,
\end{equation*}
for some $c>0$.
\end{theorem}

\begin{proof}

By the assumption, we
have
\begin{equation}\label{GInitial7new}
   \int_{r \geq r_0}(|\nabla v|^2+v^2)e^{c\sqrt{r}}e^{-2\rho}dx<\infty,
\end{equation}
for some $c>0$.
Let $q_1=\lambda-\frac{b^2}{4}-\frac{bc}{2r}-V_2$ in the energy function \eqref{Gdefenergy}.
By \eqref{Gpartial}, we have
\begin{equation}\label{Gpartialcase1new}
     \frac{\partial F(m,r,s)}{\partial r} +\alpha \frac{\partial }{\partial r}\int_{S_r}r^{s-1}v_m^2e^{-2\rho}dx
\end{equation}
\begin{eqnarray}
&=& r^{s-1}\int_{S_r}[ (r(\nabla dr) -(\frac{s}{2 }+\frac{1}{2}\bar{\delta})\hat{g})(\nabla v_m,\nabla v_m)]e^{-2\rho}dx\\
 \label{Gpartialcase2new}  &&+r^{s-1}\int_{S_r}[2m -\frac{\bar{\delta}}{2}+\frac{s}{2}]|\frac{\partial v_m}{\partial r}|^2e^{-2\rho}dx\\
  \label{Gpartialcase3new} &&+r^{s-1}\int_{S_r}[\frac{\bar{\delta} b}{2}+\bar{\delta}\frac{m} {r}+2\alpha+o(1)]\frac{\partial v_m}{\partial r}v_me^{-2\rho}dx\\
  \label{Gpartialcase4new} && +r^{s-1}\int_{S_r}[(\lambda-\frac{b^2}{4})(\frac{s}{2}+\frac{\bar{\delta}}{2})+o(1)]v_m^2e^{-2\rho}dx\\
   \label{Gpartialcase5new}&&+r^{s-1}\int_{S_r} \frac{m(m+1)}{r^2}[ \frac{s-2}{2}+\frac{\bar{\delta}}{2}]v_m^2 e^{-2\rho}dx.
\end{eqnarray}
By the assumption \eqref{GInitial7new},
  for any $m$, there exists a sequence $r_j$ such that
 \begin{equation}\label{newequa1}
   F(m,r_j,s)+\alpha \int_{S_{r_j}}r^{s-1}v_m^2e^{-2\rho}dx=o(1).
 \end{equation}
 By   assumption  \eqref{Gcons}, there exists $0<s_0<1$
 such that
 \begin{equation*}
     r(\nabla dr) -(\frac{s_0}{2 }+\frac{1}{2}\bar{\delta})>0.
 \end{equation*}
 By the Cauchy Schwartz inequality, one has
 \begin{equation*}
    \eqref{Gpartialcase2new}+\eqref{Gpartialcase3new}+\eqref{Gpartialcase4new}+\eqref{Gpartialcase5new}\geq r^{s-1}\int_{S_r}[\epsilon-C_1\frac{m^2}{r^2}]v^2_me^{-2\rho}dx,
 \end{equation*}
 for large $m$ and $r$.
 Combining with \eqref{Gpartialcase1new} and \eqref{newequa1}, we have
 \begin{equation*}
    -F(m,r,s_0)-\alpha\int_{S_r}r^{s_0-1}v_m^2e^{-2\rho}dx\geq \int_{r(x)>r} r(x)^{s_0-1}[\epsilon-C_1\frac{m^2}{r(x)^2}]v^2_me^{-2\rho}dx.
 \end{equation*}
 That is
 \begin{equation}\label{newequa3}
    r^{s_0}\int_{S_r}[\frac{1}{2}|\nabla v_m|^2-|\frac{\partial v_m}{\partial r}|^2]e^{-2\rho}dx-r^{s_0}\int_{S_r}\frac{1}{2}[\frac{m(m+1)}{r^2}+q_1]v_m^2e^{-2\rho}dx-\alpha\int_{S_r}r^{s_0-1}v_m^2e^{-2\rho}dx
 \end{equation}
 \begin{equation*}
    \geq \int_{r(x)>r} r(x)^{s_0-1}[\epsilon-C_1\frac{m^2}{r(x)^2}]v^2_me^{-2\rho}dx.
 \end{equation*}
Multiplying both side of \eqref{newequa3}  by $ r^{ 1-s_0- 2m } $ and integrating it with respect to $r$ over $[t,\infty)$,
we have
\begin{equation*}
    \int_{r>t}r^{1-2m}[|\nabla v_m|^2-(\frac{m(m+1)}{r^2}+q_1)v_m^2]e^{-2\rho}dx-2\alpha\int_{r>t}r^{-2m}v_m^2e^{-2\rho}dx-2\int_{r>t}r^{1-2m}|\frac{\partial v_m}{\partial r}|^2e^{-2\rho}dx
\end{equation*}
\begin{eqnarray}
 \nonumber  &\geq& \int_{r>t}r^{1-s_0-2m}dr  \int_{r(x)>r} r(x)^{s_0-1}[\epsilon-C_1\frac{m^2}{r(x)^2}]v^2_me^{-2\rho}dx \\
    &\geq& [\epsilon-C_1\frac{m^2}{t^2}] \int_{r>t}r^{1-s_0-2m}dr  \int_{r(x)>r} r(x)^{s_0-1}v^2_me^{-2\rho}dx \label{neqequa5}
\end{eqnarray}
By Cauchy Schwartz inequality again, one has
\begin{equation}\label{neqequa6}
   |\int_{r>t}r^{-2m}\frac{\partial v_m}{\partial r}v_me^{-2\rho}dx|\leq  \int_{r>t}\frac{1}{2}r^{-2m}[|\frac{\partial v_m}{\partial r}|^2+v_m^2]e^{-2\rho}dx.
\end{equation}
By   Lemma \ref{Keylemma3}, one has
\begin{equation}\label{Gequ35}
    \int_{r>t} r^{1-2m}[|\nabla v_m|^2-q_0v_m^2]e^{-2\rho} dx= -\frac{1}{2}\frac{d}{dt}(t^{1-2m}\int_{S_t}v_m^2e^{-2\rho}dx)-\frac{1}{2}\int_{S_t}r^{-2m}(2m-1)v_m^2e^{-2\rho}dx
\end{equation}
\begin{equation*}
  \;\;\;\;\;\;\;\;\;\;\;\;\;\;\;\;\;\;\;\;\;\;\;\;\;\;\;\;\;\;\;\;\;\;\;\;\;\;\;\;\;\;\;\;\;\;\;\; \;\;\;\;\;\;\;\;\;\; +\frac{1}{2}\int_{S_t}t^{1-2m} (\Delta r-2\rho^{\prime})v_m^2e^{-2\rho} dx-\int_{r>t}r^{-2m}\frac{\partial v_m}{\partial r}v_me^{-2\rho}dx,
\end{equation*}
where  $q_0=\lambda-V_0-V_1-V_2+\frac{m(m+1)}{r^2}+\frac{m}{r}(2\rho^{\prime}-\Delta r)$.

By the definition of $q_0$ and $q_1$, we have
\begin{equation}\label{neqequa7}
    |q_0-q_1-\frac{m(m+1)}{r^2}|\leq 2\frac{m}{r^2}|\delta|+O(\frac{1}{r}).
\end{equation}
By  \eqref{neqequa5}, \eqref{neqequa6}, \eqref{Gequ35} and  \eqref{neqequa7},
we have
\begin{equation*}
     -\frac{1}{2}\frac{d}{dt}(t^{1-2m}\int_{S_t}v_m^2e^{-2\rho}dx)-\frac{1}{4}\int_{S_t}t^{-2m}(2m-1)v_m^2e^{-2\rho}dx -\frac{1}{2}\alpha\int_{r>t}r^{-2m}v_m^2e^{-2\rho}dx+C_2\int_{r>t} r^{-2m} \frac{m}{r}v_m^2e^{-2\rho}dx
\end{equation*}
\begin{equation}\label{neqequa9}
   \geq [\epsilon-C_1\frac{m^2}{t^2}] \int_{r>t}r^{1-s_0-2m}dr  \int_{r(x)>r} r(x)^{s_0-1}v^2_me^{-2\rho}dx ,
\end{equation}
for large $m$ and $t$.
Notice that $C_1$ and $C_2$ do not depend on $\alpha$.
For any large $t$, let
  $m$ be such that
\begin{equation*}
    \epsilon- C_1\frac{m^2}{t^2}=0.
\end{equation*}
That is $\frac{m}{t}=\sqrt{\frac{\epsilon}{C_1}}$.

Let $\alpha$ be large enough so that
\begin{equation*}
      -\frac{1}{2}\alpha\int_{r>t}r^{-2m}v_m^2e^{-2\rho}dx+C_2\int_{r>t} r^{-2m} \frac{m}{r}v_m^2e^{-2\rho}dx<0,
\end{equation*}
since $\frac{m}{r}\leq \frac{m}{t}$ for $r\geq t$.
Thus, \eqref{neqequa9} leads to
\begin{equation*}
  -\frac{1}{2}\frac{d}{dt}(t^{1-2m}\int_{S_t}v_m^2e^{-2\rho}dx)-\frac{1}{4}\frac{m}{t}\int_{S_t}t^{1-2m}v_m^2e^{-2\rho}dx <0.
\end{equation*}
By the fact that $\frac{m}{t}=\sqrt{\frac{\epsilon}{C_1}}$,
we have
\begin{equation}\label{last1}
  -\frac{d}{dt}(t^{1-2m}\int_{S_t}v_m^2e^{-2\rho}dx)-\epsilon\int_{S_t}t^{1-2m}v_m^2e^{-2\rho}dx \geq 0
\end{equation}
for some $\epsilon>0$.
Let
\begin{equation*}
    G(t)=t^{1-2m}\int_{S_t}v_m^2e^{-2\rho}dx=t\int_{S_t}v^2e^{-2\rho}dx.
\end{equation*}
So \eqref{last1} becomes
\begin{equation*}
    \frac{dG}{dt}\leq -\epsilon G(t)
\end{equation*}
This implies
\begin{equation*}
    G(t)\leq e^{-\epsilon t}
\end{equation*}
for some $ \epsilon>0$.
Thus,
we get
\begin{equation}\label{last3}
    \int_{S_r}v^2e^{-2\rho}dx\leq e^{-\epsilon r}.
\end{equation}
We will show that
\begin{equation}\label{last7}
    \int_{r \geq r_0}|\nabla v|e^{-2\rho}e^{\epsilon r} dx <\infty
\end{equation}
for some $\epsilon>0$.
By Cauchy Schwartz inequality and \eqref{GInitial7new},
one has
\begin{equation}\label{last8}
    \int_{r \geq r_0}|\frac{\partial v}{\partial r}v| e^{\epsilon r}e^{-2\rho}dx<\infty.
\end{equation}
Thus there exists sequence $r_j$
such that
\begin{equation*}
    \int_{S_{r_j}}|\frac{\partial v}{\partial r}v|e^{\epsilon r}e^{-2\rho}dx=o(1).
\end{equation*}
Using the eigenequation \eqref{equav} and integration by part, one has
\begin{equation}\label{last535}
    \int_{r_0\leq r \leq r_j}  |\nabla v|^2 e^{\epsilon r}  e^{-2\rho}dx+ \int_{S_{r}}\frac{\partial v}{\partial r}ve^{\epsilon r} e^{-2\rho}dx-\int_{S_{r_j}}\frac{\partial v}{\partial r}v e^{\epsilon r}e^{-2\rho}dx<\infty,
\end{equation}
since \eqref{last3} and \eqref{last8} hold.

By letting $j\to \infty$ in \eqref{last535}, we obtain
\begin{equation*}
    \int_{ r >r_0}  |\nabla v|^2 e^{\epsilon r}  e^{-2\rho}dx<\infty.
\end{equation*}
We finish the proof.

\end{proof}
\begin{theorem}\label{Expthm2}
Suppose 
\begin{equation}\label{GInitial7}
   \int_{r\geq r_0}(|\nabla v|^2+v^2)e^{c {r}^{\sigma}}e^{-2\rho}dx<\infty,
\end{equation}
for any $\sigma<1$, then  $v\equiv0$.
\end{theorem}
\begin{proof}
Let $\bar{v}=e^{kr^{\theta}}v$ and $\theta$ be close to 1.
In order to write down our proof smoothly, let $\bar{\rho}=kr^{\theta}$.
Direct computation implies
\begin{equation*}
    \Delta v= e^{-\bar{\rho}}\Delta \bar{v} -2\bar{\rho}^{\prime}e^{-\bar{\rho}}\frac{\partial \bar{v}}{\partial r}+(\bar{\rho}^{\prime 2}-\bar{\rho}^{\prime \prime}-\bar{\rho}^{\prime} \Delta r)e^{-\bar{\rho}}\bar{v}.
\end{equation*}
By \eqref{equav}, we get the equation of $\bar{v}$,
\begin{equation*}
    \Delta \bar{v}=2(\bar{\rho}^{\prime}+\rho^{\prime})\frac{\partial \bar{v}}{\partial r}+(V_1+V_2+V_0+\bar{V}_0-\lambda)\bar{v},
\end{equation*}
where
\begin{eqnarray*}
   \bar{V}_0 &=& -\bar{\rho}^{\prime 2}+\bar{\rho}^{\prime}\Delta r+\bar{\rho}^{\prime\prime}-2\rho^{\prime}\bar{\rho}^{\prime} \\
   &=& -k^2\theta^2r^{2\theta-2}-k(1-\theta)\theta r^{\theta-2}+kr^{\theta-1}(\Delta r-2\rho^{\prime}).
\end{eqnarray*}

Define the  new energy function of $\bar{v}$ as follows
\begin{equation}\label{GInitial835}
    \bar{F}(r)=\int_{S_r} [\frac{1}{2}\bar{q}\bar{v}^2+|\frac{\partial \bar{v}}{\partial r}|^2-\frac{1}{2}|\nabla \bar{v}|^2]e^{-2\rho}dx,
\end{equation}
where $\bar{q}$ will be determined later.
By the assumption \eqref{Gcons}, there exists some $s_0$ (actually we can let  $s_0$ be close to $\delta$ and $s_0>\delta$) such that
\begin{equation*}
   | r(\Delta r-2\rho^{\prime})+s_0|<2a
\end{equation*}
and
\begin{equation*}
    r(\Delta r-2\rho^{\prime})+s_0>\gamma>0.
\end{equation*}

By the similar argument of \eqref{Gpartial} or  \eqref{Gpartialnew}, one has
\begin{eqnarray}
   \nonumber\frac{\partial r^{s_0} \bar{F}(r)}{\partial r}&=&\int_{S_r} r^{s_0}[( \nabla dr-(\frac{s}{2r}+\frac{1}{2}(\Delta r-2\rho^{\prime}))\hat{g})(\nabla \bar{v},\nabla\bar{ v})]e^{-2\rho} dx\\
   \nonumber &&+\int_{S_r} r^{s_0}[\bar{q}\bar{v}\frac{\partial \bar{v}}{\partial r}+\frac{1}{2}\frac{\partial \bar{q}}{\partial r }\bar{v}^2+\frac{\partial \bar{v}}{\partial r}(\Delta  \bar{v}-\Delta r\frac{\partial \bar{v}}{\partial r})+(\Delta r-2\rho^{\prime})(\frac{1}{2}\bar{q}\bar{v}^2+\frac{1}{2}|\frac{\partial \bar{v}}{\partial r}|^2)]e^{-2\rho}dx\\
  \nonumber  &&+s_0r^{s_0-1}\int_{S_r}[\frac{1}{2}\bar{q}\bar{v}^2+\frac{1}{2}|\frac{\partial \bar{v}}{\partial r}|^2]e^{-2\rho}dx\\
   &=&r^{s_0-1}\int_{S_r} [(r \nabla dr-( \frac{s_0}{2}+\frac{1}{2}\bar{\delta})\hat{g})(\nabla \bar{v},\nabla\bar{ v})]e^{-2\rho} dx\label{last135}\\
  \nonumber  &&+r^{s_0-1}\int_{S_r}[-\frac{\bar{\delta}}{2}+\frac{s_0}{2}+2r\bar{\rho}^{\prime}]|\frac{\partial \bar{v}}{\partial r}|^2 e^{-2\rho}dx\\
   \nonumber &&+r^{s_0-1}\int_{S_r}r[\bar{q}+V_1+V_2+V_0+\bar{V}_0-\lambda]\frac{\partial \bar{v}}{\partial r}\bar{v}  e^{-2\rho}dx\\
  \nonumber  && +r^{s_0-1}\int_{S_r} [\frac{r}{2}\frac{\partial \bar{q}}{\partial r}+\frac{1}{2}\bar{\delta}\bar{q}+\frac{s_0}{2}\bar{q}]\bar{v}^2e^{-2\rho}dx.
\end{eqnarray}
Let $\bar{q}=\lambda-\frac{b^2}{4}-V_2+k^2\theta^2r^{2\theta-2}+k\theta(1-\theta)r^{\theta-2}$.
By \eqref{last135}, one has
\begin{eqnarray}
   \nonumber\frac{\partial r^{s_0} \bar{F}}{\partial r}&=&r^{s_0-1}\int_{S_r} [( r\nabla dr-( \frac{s_0}{2}+\frac{1}{2}\bar{\delta})\hat{g})(\nabla \bar{v},\nabla\bar{ v})]e^{-2\rho} dx\\
   &&+r^{s_0-1}\int_{S_r}[2k\theta r^{\theta}+O(1)]|\frac{\partial \bar{v}}{\partial r}|^2 e^{-2\rho}dx\label{theta1}\\
   &&+r^{s_0-1}\int_{S_r}[O(1)+k\bar{\delta}r^{\theta-1}]\frac{\partial \bar{v}}{\partial r}\bar{v}  e^{-2\rho}dx\label{theta2}\\
   && +r^{s_0-1}\int_{S_r} [o(1)+k^2\theta^2(\theta-1)r^{2\theta-2}+\frac{1}{2}k\theta(1-\theta)(\theta-2)r^{\theta-2}\label{theta4}\\
   &&+\frac{1}{2}(\bar{\delta}+s_0)(\epsilon+o(1)+k(1-\theta)\theta r^{\theta-2}+k^2\theta^2r^{2\theta-2})]\bar{v}^2e^{-2\rho}dx\label{theta3}
\end{eqnarray}
with some $ \epsilon=\lambda-\frac{b^2}{4}>0$.
By the choice that $\theta$ is close to 1, one has
\begin{equation*}
  \eqref{theta4}+\eqref{theta3}>0
\end{equation*}
for large $r$ and $k$.

By Cauchy Schwartz inequality, we have
\begin{equation*}
   | \eqref{theta2}|<\eqref{theta1}+\eqref{theta4}+\eqref{theta3}.
\end{equation*}
Thus
\begin{equation}\label{last9}
    \frac{\partial r^{s_0}\bar{F}(r)}{\partial r}>0,
\end{equation}
for large $k$ (only depends on the bounds of constants in the assumptions of Theorem \ref{Mainthm1}).

By the assumption of \eqref{GInitial7}, there exists sequence  $r_j$ going to infinity
such that
\begin{equation*}
   \lim_{j} r_j^{s_0}\bar{F}(r_j)=0.
\end{equation*}
By the fact that     $ r^{s_0}\bar{F}(r)$ is monotone (using \eqref{last9}),
we have
\begin{equation*}
    \bar{F}(r)\leq 0
\end{equation*}
for $r> R$. Note that $R$ does not depend on $k$.
Fix  $r>R$, by \eqref{GInitial835}, one has
\begin{equation*}
    \bar{F}(r)=e^{2kr^{\theta}}\int_{S_{r}} [\frac{1}{2}(\bar{q}+k^2\theta^2r^{2\theta-2}) {v}^2+|\frac{\partial  {v}}{\partial r}|^2+k\theta r^{\theta-1}\frac{\partial  {v}}{\partial r} {v}-\frac{1}{2}|\nabla  {v}|]e^{-2\rho}dx\leq 0,
\end{equation*}
for all large $k$ and large $r$.
This is impossible except
\begin{equation*}
    \int_{S_r}|v|^2 e^{-2\rho}dx=0
\end{equation*}
for all large $r$.
By unique continuation theorem, we have $v\equiv0$.
\end{proof}
\section{The positivity of initial energy and proof of   Theorem \ref{Mainthm1}  }\label{proof1}
\begin{theorem}\label{thminitial}
Suppose
\begin{equation}\label{Gq}
q-\lambda+V_0+V_1+V_2\geq \frac{\epsilon}{r},
\end{equation}
for some $\epsilon>0$.
 We also assume
\begin{equation}\label{Ginitial2}
    \liminf_{r}\int_{S_r}|\frac{\partial v}{\partial r}v|e^{-2\rho}dx=0.
\end{equation}
Then for any $R>0$, there exists $r>R$
such that
\begin{equation*}
    F(r,0)>0.
\end{equation*}

\end{theorem}
\begin{proof}

Otherwise $ F(r,0)\leq 0$ for all $r>R$.

By the assumption \eqref{Ginitial2}, there exists a sequence $r_j$ going to infinity such that
\begin{equation}\label{Gequ8}
    \int_{S_{r_j}}|\frac{\partial v}{\partial r}v|e^{-2\rho}dx=o(1).
\end{equation}
Integrating $F(r,0)$  from $r$ to $r_j$,
one has
\begin{equation}\label{Gequ6}
    \int_{r\leq r(x)\leq r_j}[qv^2 +|\frac{\partial v}{\partial r}|^2-|\nabla_\omega v|^2] e^{-2\rho}dx\leq 0.
\end{equation}
Using the eigen-equation \eqref{equav} and integration by part, one has
\begin{equation}\label{Gequ7}
    \int_{r\leq r(x)\leq r_j}  |\nabla v|^2   e^{-2\rho}dx+ \int_{S_{r}}\frac{\partial v}{\partial r}v e^{-2\rho}dx-\int_{S_{r_j}}\frac{\partial v}{\partial r}v e^{-2\rho}dx=\int_{r\leq r(x)\leq r_j} [\lambda-V_0-V_1-V_2]v^2e^{-2\rho}dx.
\end{equation}
Putting    \eqref{Gequ8} and \eqref{Gequ7} into  \eqref{Gequ6}, and letting $r_j$ go to infinity, we have
\begin{equation*}
    \int_{S_{r}}\frac{\partial v}{\partial r}v e^{-2\rho}dx +\int_{r(x)\geq r}[(q-\lambda+V_0+V_1+V_2)v^2+ 2|\frac{\partial v}{\partial r}|^2 ]e^{-2\rho}dx\leq 0.
\end{equation*}
By the assumption \eqref{Gq}, we have
\begin{equation}\label{GInitial1235}
    \int_{S_{r}}\frac{\partial v}{\partial r}v e^{-2\rho}dx +\int_{r(x)\geq r}[\frac{\epsilon}{r}v^2+ 2|\frac{\partial v}{\partial r}|^2 ]e^{-2\rho}dx\leq 0.
\end{equation}
Let
\begin{equation*}
   G(r)=\int_{r(x)\geq r}[ \frac{\epsilon}{r}v^2+ 2|\frac{\partial v}{\partial r}|^2 ]e^{-2\rho}dx.
\end{equation*}
By \eqref{GInitial1235} and Cauchy-Schwartz inequality,
one has
\begin{equation}\label{GInitial2}
  G(r)\leq -C \sqrt{r}  \frac{d G}{dr}.
\end{equation}
where $C>0$ is a constant depending on $\epsilon$.
Solving \eqref{GInitial2}, one has
\begin{equation*}
    G(r)\leq e^{-c \sqrt{r}}
\end{equation*}
for large $r$.
This yields that
\begin{equation}\label{GInitial5}
  \int_{r \geq r_0}[ v^2+ 2|\frac{\partial v}{\partial r}|^2 ]e^{c\sqrt{r}}e^{-2\rho}dx<\infty,
\end{equation}
for  some $c>0$,
and
\begin{equation*}
  \int_{r\geq r_0}|v\frac{\partial v}{\partial r}| e^{c\sqrt{r}}e^{-2\rho}dx<\infty.
\end{equation*}
Multiplying $e^{c\sqrt{r}} v e^{-2\rho}$  in \eqref{equav} and integration by part,
one has
\begin{equation}\label{GInitial34}
   \int_{r \geq r_0}|\nabla v|^2e^{c\sqrt{r}}e^{-2\rho}dx<\infty.
\end{equation}
By  Theorems \ref{Expthm1}, \ref{Expthm2} and \eqref{GInitial5}, \eqref{GInitial34}, we have
$v\equiv 0$, which is contradicted to our assumption.

\end{proof}

Recall that
\begin{equation*}
     \bar{\delta} (r) = r(\Delta r-b-\frac{c}{r}).
\end{equation*}
Let $q=\lambda-\frac{b^2}{4}-\frac{bc}{2r}+\frac{b(\delta\pm\epsilon)}{2r}-V_2$, where $\pm$ is defined such that $b(\pm\epsilon)>0$.
Thus \eqref{Gq} holds.
Now by \eqref{Gpartialnew}, the derivative of the energy function for $v$ has the following form,
\begin{eqnarray}
  \label{Gpartialcase135} \frac{\partial F( r,s)}{\partial r} 
   &=& r^{s-1}\int_{S_r}[ r(\nabla dr) -(\frac{s}{2 }+\frac{1}{2}\bar{\delta})\hat{g})(\nabla v,\nabla v)]e^{-2\rho}dx\\
 \label{Gpartialcase235} &&+r^{s-1}\int_{S_r}[  -\frac{\bar{\delta}}{2}+\frac{s}{2}+o(1)]|\frac{\partial v}{\partial r}|^2e^{-2\rho}dx\\
  \label{Gpartialcase335} &&+r^{s-1}\int_{S_r}[ \frac{\bar{\delta} b}{2} +\frac{ {(\delta\pm\epsilon)} b}{2}]\frac{\partial v}{\partial r}ve^{-2\rho}dx\\
  \label{Gpartialcase435} && +r^{s-1}\int_{S_r}[(\lambda-\frac{b^2}{4})(\frac{s}{2}+\frac{\bar{\delta}}{2})+o(1)]v^2e^{-2\rho}dx.
\end{eqnarray}
\begin{theorem}\label{Thmpositivezero}
Let  $s<\mu$ and $s$ be sufficiently close to $\mu$. Then
under the assumptions  of Theorem \ref{Mainthm1},
 we have
 \begin{equation*}
    \frac{\partial F( r,s)}{\partial r} >0,
 \end{equation*}
 for large $r$.
\end{theorem}
\begin{proof}
By the assumptions of Theorem \ref{Mainthm1} and $s$ is close to $\mu$, one has
\begin{equation*}
   |\bar{\delta}+s|<2a
\end{equation*}
and
\begin{equation*}
   \bar{\delta}+s >0.
\end{equation*}
By Cauchy Schwartz inequality and \eqref{Gpartialcase135}-\eqref{Gpartialcase435}, it suffices to prove
\begin{equation}\label{Gsolveeq}
    4[(\lambda-\frac{b^2}{4})(\frac{s}{2}+\frac{\bar{\delta}}{2})][-\frac{\bar{\delta}}{2}+\frac{s}{2}]>
    |\frac{\bar{\delta} b}{2} +\frac{ {(\delta\pm\epsilon)} b}{2}|^2.
\end{equation}
Solving inequality \eqref{Gsolveeq},
we get
\begin{equation}\label{Gsolveeq1}
     \lambda>\frac{b^2}{4}+\frac{1}{4}\frac{(\bar{\delta}+\delta\pm\epsilon)^2b^2}{s^2-\bar{\delta}^2}.
\end{equation}
It is clear that \eqref{Gsolveeq1} holds if
\begin{equation*}
     \lambda>\frac{b^2}{4} + \frac{\delta^2b^2}{s^2-{\delta}^2},
\end{equation*}
since $\epsilon$ is small.
This follows from  the  assumption \eqref{Gconl} and the fact that $s$ is close to $\mu$.
\end{proof}

{\textbf{Proof of Theorem \ref{Mainthm1}}.}
\begin{proof}
It suffices to  assume
\begin{equation*}
 \liminf_{r}\int_{S_r}|\frac{\partial v}{\partial r}v|e^{-2\rho}dx=0
\end{equation*}
so that \eqref{Ginitial2} holds.
By Theorem \ref{thminitial}, there exists a sequence $r_n$ such that
\begin{equation*}
    F(r_n,0)>0.
\end{equation*}
By Theorem \ref{Thmpositivezero},
 there exists $\gamma>0$ such that
\begin{equation*}
    F( r, s)\geq\gamma,
\end{equation*}
for large $r$.
By the fact  $v=e^{\rho}u$, we get that
\begin{equation*}
   \liminf_{r\to \infty} r^{s} [M(r)^2+N(r)^2]>0 .
\end{equation*}
Recall that   $s<\mu$,  we have
\begin{equation*}
   \liminf_{r\to \infty} r^{\mu} [M(r)^2+N(r)^2]=\infty .
\end{equation*}
\end{proof}

 \section*{Acknowledgments}

I would like to thank Svetlana Jitomirskaya for introducing to me paper \cite{kumura2010radial} and
inspiring discussions on this subject. I would also like to thank    the  anonymous referee for careful reading of the manuscript that has led to an important improvement, in particular, telling me the  work of Ito and Skibsted.
Finally, I  thank Professor Ito for helpful discussions.
  The author  was supported by the AMS-Simons Travel Grant 2016-2018 and  NSF DMS-1700314. This research was also
partially supported by NSF DMS-1401204.


\end{document}